\documentclass[final]{amsart}
\usepackage{amssymb}
\usepackage{hhline}
\usepackage{graphicx}
\usepackage{stmaryrd}
\usepackage{float}
\usepackage{cite}
\usepackage{orcidlink}

\usepackage[utf8]{inputenc}
\usepackage{listings}
\usepackage{xcolor}

\newtheorem{theorem}{Theorem}[section]

\newtheorem{proposition}[theorem]{Proposition}
\newtheorem{corollary}[theorem]{Corollary}

\newtheorem{problem}[theorem]{Problem}
\theoremstyle{definition}

\newcommand{\Aut}{\mathrm{Aut\mkern 2mu}}
\newcommand{\Halo}{\mathrm{Halo\mkern 2mu}}
\newcommand{\M}{\mathrm{M\mkern 1mu}}

\title{On  dimonoids of small order}

\author{Volodymyr M. Gavrylkiv\,\orcidlink{0000-0002-6256-3672}}
\address[V.~Gavrylkiv]{Vasyl Stefanyk Carpathian National University, Ivano-Frankivsk, Ukraine} \email{vgavrylkiv@gmail.com}
\subjclass{18B40, 37L05, 22A15, 20D45, 20M15, 20B25}
\keywords{semigroup, dimonoid, abelian dimonoid, commutative dimonoid, rectangular dimonoid, halo, automorphism group}

\begin{document}

\begin{abstract}
This paper is devoted to the study of dimonoids, i.e., algebraic structures equipped with two associative binary operations satisfying a given system of axioms. We introduce several new classes of dimonoids and determine their automorphism groups and halos. In particular, we construct examples of iso-dual nonabelian dimonoids, as well as nonabelian dimonoids with nonempty halo. Using these constructions, we obtain a complete classification, up to isomorphism, of all nonabelian noncommutative nontrivial dimonoids of order~$3$, thereby resolving the problem concerning this classification. Consequently, all three-element dimonoids are classified up to isomorphism, yielding exactly $52$ pairwise nonisomorphic dimonoids of order~$3$. Finally, we present the results of computer computations determining the numbers of all pairwise nonisomorphic dimonoids of orders up to~$5$, as well as all pairwise nonisomorphic commutative, abelian, and rectangular dimonoids of orders up to~$6$, obtained using \texttt{GAP}, \texttt{Python}, and \texttt{C++}.
\end{abstract}

\maketitle

\section*{Introduction}

The concept of  dialgebra was introduced by Jean-Louis Loday~\cite{Lod} in his search for a class of (linear) algebras that generate Leibniz algebras in a manner analogous to how associative algebras give rise to Lie algebras. Recall that a Leibniz algebra is a linear algebra over a field whose bracket operation $[, ]$ satisfies the Leibniz derivation identity
$$[[x, y], z] = [[x, z], y] + [x, [y, z]],$$
without necessarily being anticommutative.

Loday’s idea was to “separate” the left and right multiplications, treating them as distinct associative operations, denoted $\vdash$ and $\dashv$.
He observed that if these operations satisfy the following three axioms:
\begin{align*}
(x \dashv y) \dashv z &= x \dashv (y \vdash z), \hspace{10mm}(D_1) \\
(x \vdash y) \dashv z &= x \vdash (y \dashv z), \hspace{10mm}(D_2)\\\
(x \dashv y) \vdash z &= x \vdash (y \vdash z), \hspace{10mm}(D_3)
\end{align*}
then the bracket defined by $[x, y] = x \dashv y - y \vdash x$ satisfies the Leibniz identity.

Accordingly, Loday defined a {\em dimonoid}~\cite{Lod} as an algebraic structure $(D,\dashv,\vdash)$ consisting of a set $D$ equipped with two associative binary operations, $\dashv$ and $\vdash$, satisfying axioms $(D_1)$, $(D_2)$, and $(D_3)$. Each semigroup $(D,\dashv)$ can naturally be regarded as a dimonoid $(D,\dashv,\dashv)$, called the {\em trivial dimonoid}. In this sense, dimonoids generalize semigroups. Since dialgebras are the linear analogues of dimonoids, many results on dimonoids have direct applications in the theory of dialgebras~\cite{B,F,Lod,M,ZCY}. Given the central role of dimonoids in the study of Leibniz algebras, their investigation by semigroup-theoretic methods represents a natural and promising direction of research.

T.~Pirashvili~\cite{P} introduced the concept of a duplex, a generalization of dimonoids, and constructed the free duplex. In~\cite{MDS}, a generalized dimonoid (or $g$-dimonoid) is introduced as an algebraic structure similar to a dimonoid, satisfying all the axioms of a dimonoid except possibly $(D_2)$, and the construction of the free $g$-dimonoid is also described. Several classes of $g$-dimonoids were studied in~\cite{Ggdim1}. The properties of free dimonoids were employed in~\cite{Lod} to characterize free dialgebras and to study their cohomologies. In~\cite{Liu}, the notion of a dimonoid was used to define and investigate one-sided dirings. Furthermore, dimonoids are closely related to restrictive bisemigroups~\cite{Sh} and doppelsemigroups~\cite{GR1, GDS2, GUDS, GSDS, GLDS, GDSso, Zh2017AU}.

One of the earliest foundational results on dimonoids is due to Loday~\cite{Lod}, who provided a description of the absolutely free dimonoid generated by a given set.
A wide range of classes of dimonoids have been systematically investigated by Anatolii Zhuchok and Yurii Zhuchok. In~\cite{ZhAL2011}, the independence of the dimonoid axioms was established. Commutative, free commutative, and free abelian dimonoids were studied in~\cite{ZhADM2009}, \cite{ZhADM2010}, and~\cite{ZhADM2015}, respectively. The structure of dibands of subdimonoids and semilattice decompositions of dimonoids was explored in~\cite{ZhMS2011, ZhDM2011}. Free rectangular dimonoids, as well as free normal and free $(\ell r, rr)$-dibands, were constructed in~\cite{ZhADM2011a}, \cite{ZhADM2011}, and~\cite{ZhADM2013}, respectively. Free abelian dibands and some of their properties were studied in~\cite{ZhVLU2017, ZhADM2018}. The least semilattice congruence on free dimonoids was described in~\cite{ZhUMJ2011}. Free products of dimonoids and relatively free dimonoids were the focus of several works, including~\cite{ZhQRL2013, ZhMP2014, ZhCA2017a, ZhIJAC2021}. Moreover, the free left $n$-nilpotent and free left $n$-dinilpotent dimonoids were constructed in~\cite{ZhADM2013nil, ZhSF2016}. Representations of ordered dimonoids via binary relations were examined in~\cite{ZhAEJM2014}.
Significant contributions to the theory of endomorphisms and automorphisms in the context of dimonoids were made by Yu.~Zhuchok in~\cite{ZhBAS2014,ZhCA2017,ZhJA2024}.

In~\cite{Gdim1}, we studied the properties of dual dimonoids and, within the class of noncommutative dimonoids, constructed various examples of abelian, nonabelian, and rectangular dimonoids. The algebraic structure of these dimonoids was examined in detail, including computations of their automorphism groups and halos. In~\cite{Gdim2}, we built upon these results, as well as those obtained in~\cite{GR1, GDS2}, to provide a complete classification, up to isomorphism, of all two-element dimonoids, all commutative three-element dimonoids, and all abelian three-element dimonoids. 

In the present paper, we introduce new classes of dimonoids and determine their automorphism groups and halos. Using these constructions, we obtain a complete classification, up to isomorphism, of all nonabelian noncommutative nontrivial dimonoids of order~$3$, thereby resolving Problem~4.6 posed in~\cite{Gdim2} concerning this classification. As a consequence, we also provide a complete classification, up to isomorphism, of all three-element dimonoids. Finally, we present the results of computer computations determining the numbers of all pairwise nonisomorphic dimonoids of orders up to~$5$, as well as all pairwise nonisomorphic commutative, abelian, and rectangular dimonoids of orders up to~$6$. 

\section{Preliminaries on semigroups}

We shall follow the standard definitions and terminology of semigroup theory (see, for example, \cite{Howie}).

\smallskip

A semigroup $(S,*)$ is called a {\em null semigroup} if there exists an element $0\in S$ such that $x*y=0$ for all $x,y\in S$. In this case  $0$ is a zero of $S$.  By $O_{S^0}$ we denote a null semigroup with zero $0$ on a set $S$.  The null semigroups $O_{S^0}$ and $O_{T^z}$ are isomorphic if and only if $|S|=|T|$. If $S$ is a set of cardinality $|S|=n$, we use the notation $O_n$ for a representative of the class of semigroups isomorphic to $O_{S^0}$.

\smallskip

A semigroup $(S, *)$ is called \emph{rectangular}~\cite{ZhCA2017a} if every element of $S$ is a middle identity, that is,
$x*y*z = x*z$ for all  $x, y, z \in S$. In other words, the product in $S$ depends only on the first and the last factors. 
A nontrivial null semigroup is an  example of a commutative rectangular semigroup that has no identity element.
In~\cite{Nagy}, it was proved that a semigroup $(S,*)$ is rectangular if and only if the factor semigroup $S/\theta$ is a right zero semigroup, 
where $\theta = \{(a, b)\in S\times S : x*a = x*b \text{ for all } x\in S\}$. The authors also described a method for constructing all rectangular semigroups. In~\cite{GS_sie}, the study focused on rectangular semigroups, including rectangular ideal extensions of left (right) zero semigroups by null quotients, and provided new combinatorial  results for the numbers of pairwise nonisomorphic instances.

\smallskip

Let $S$ be a nonempty set, $0\in S$ and $A\subset S\setminus\{0\}$. Define  the binary operation $*$ on  $S$ in the following way:

$$x* y=\begin{cases}
x, \text{ if }y=x\in A \\
0,\ \text{otherwise}.
\end{cases}$$

It is easy to check that a set $S$ endowed with the operation $*$ is a commutative semigroup with  zero $0$, and we denote this semigroup by $O^A_{S^0}$. If $A=S\setminus\{0\}$, then $O^A_{S^0}$ is a semilattice. In the case when $A$ is an empty set,  $O^A_{S^0}$ coincides with a null  semigroup  with  zero $0$. The semigroups $O^A_{S^0}$ and $O^B_{T^z}$ are isomorphic if and only if $|S|=|T|$ and $|A|=|B|$. If $S$ is a finite set of cardinality $|S|=n$ and $|A|=m$,  we use the notation $O_n^m$ for a representative of the class of semigroups isomorphic to $O^A_{S^0}$.

\smallskip

If $(S,*)$ is a semigroup, then the semigroup $(S,{*}^d)$ with operation $x{*}^d y=y* x$ is called {\em dual} to $(S,*)$, denoted $(S,*)^d$. It follows that $(S,*)^d = (S,*)$ if and only if $(S,*)$ is a commutative semigroup, and $\Aut(S,*)^d=\Aut(S,*)$.

\smallskip

A semigroup $(S,*)$ is said to be a {\em left} (resp. {\em right}) {\em zero semigroup} if $a*b=a$ (resp. $a*b=b$) for any $a,b\in S$. We denote by  $LO_S$ and $RO_S$ a left zero semigroup and a right zero semigroup on a set $S$, respectively.  The semigroups $LO_S$ and $RO_S$ are dual rectangular bands. If $S$ is a  set of cardinality $|S|=n$,  we use the notations $LO_n$ and $RO_n$ for representatives of the classes of semigroups isomorphic to $LO_S$ and $RO_S$, respectively.

\smallskip

Let $S$ be a nonempty set, $A\subset S$ and $0\notin S$. Define  the binary operation $*$ on  $S^{\sim 0}=S\cup\{0\}$ in the following way:

$$x* y=\begin{cases}
x, \text{ if }y \in A \\
0, \text{ if }y \in S^{\sim 0}\setminus A.
\end{cases}$$

It is easy to check that a set $S^{\sim 0}$ endowed with the operation $*$ is a semigroup with  zero $0$, and we denote this semigroup by $LO^{\sim 0}_{A\leftarrow S}$. If  $A$ is an empty set,  then $LO^{\sim 0}_{A\leftarrow S}$ coincides with a null semigroup $O_{S^0}$ with  zero $0$. The semigroups $LO^{\sim 0}_{A\leftarrow S}$ and $LO^{\sim z}_{B\leftarrow T}$ are isomorphic if and only if $|S|=|T|$ and $|A|=|B|$. If $S$ is a finite set of cardinality $|S|=n$ and $|A|=m$,  we use the notation $LO^{\sim 0}_{m\leftarrow n}$ for a representative of the class of semigroups isomorphic to $LO^{\sim 0}_{A\leftarrow S}$.

By $RO^{\sim 0}_{A\leftarrow S}$ we denote a dual semigroup of $LO^{\sim 0}_{A\leftarrow S}$ and use the notation $RO^{\sim 0}_{m\leftarrow n}$ accordingly.
 
\smallskip 

Let  $a$ and $c$ be different elements of a set $S$. Define the associative binary operation $\dashv_c^a$ on  $S$ in the following way:

$$x\dashv_c^a y=\begin{cases}
a,\text{ if } x=y=a \\
c,\text{ if } x=a\text{ and } y\neq a\\
x,\text{ if } x\neq a.
\end{cases}$$

If $|S|\geq 3$, then $(S,\dashv_c^a)$ is a noncommutative band in which all elements $z\neq a$ are left zeros. We denote this band by $LOB_S^{a,c}$.
It is straightforward to verify that for any different $b, d\in S$, the semigroups $(S,\dashv_c^a)$ and $(S,\dashv_d^b)$ are
isomorphic. If $S$ is a finite set of cardinality $|S|=n$, we use the notation $LOB_n$ for a representative of the class of semigroups isomorphic to $LOB_S^{a,c}$. 

By $ROB_S^{a,c}$ we denote a dual semigroup of $LOB_S^{a,c}$ and use the notation $ROB_n$ accordingly.

\smallskip

Let $S$ be a nonempty set, $A \subseteq S$ a nonempty subset, and $a \in A$.
Define an associative binary operation $*$ on $S$ as follows:

$$x* y=\begin{cases}
x,\ \text{ if } x\in A \\
a,\ \text{ if } x \notin A.
\end{cases}$$

We denote the semigroup $(S,*)$  by $LO_{A_a\leftarrow S}$. It follows that  all elements $z\in A$ are left zeros of $LO_{A_a\leftarrow S}$. If $A=\{a\}$, then $LO_{A_a\leftarrow S}$ coincides with a null semigroup $O_{S^a}$ with zero $a$. If $A=S$, then $LO_{A_a\leftarrow S}$ coincides with a left zero semigroup $LO_S$. The semigroups $LO_{A_a\leftarrow S}$ and $LO_{B_b\leftarrow T}$  are isomorphic if and only if $|S|=|T|$ and $|A|=|B|$. If $S$ is a finite set of cardinality $|S|=n$ and $|A|=m$,  we use the notation $LO_{m\leftarrow n}$ for a representative of the class of semigroups isomorphic to  $LO_{A_a\leftarrow S}$.

By $RO_{A_a\leftarrow S}$ we denote a dual semigroup of $LO_{A_a\leftarrow S}$ and use the notation $RO_{m\leftarrow n}$ accordingly. Both $LO_{A_a\leftarrow S}$ and $RO_{A_a\leftarrow S}$ are rectangular semigroups. Each of them is a rectangular band precisely when $A=S$.

\bigskip

Following the algebraic tradition, we take for a model of the class of cyclic groups of order $n$ the multiplicative group
$C_n=\{z\in\mathbb C:z^n=1\}$ of $n$-th roots of $1$. For a set $X$, we denote by $S_X$  the group of all permutations of $X$.

\section{Some definitions and basic properties of dimonoids}

In this section, we recall several useful results on dimonoids  that will be  used in the subsequent investigations.

An element $e$ of a dimonoid $(D,\dashv, \vdash)$ is called a {\em bar-unit}~\cite{Lod} if $e \vdash d = d = d \dashv e$ for all $d\in D$. In contrast to monoids a dimonoid may have many bar-units. The set of all bar-units of a dimonoid $(D,\dashv, \vdash)$ is called the {\em halo} of $(D,\dashv, \vdash)$, denoted $\Halo(D,\dashv, \vdash)$. A nonempty subset $B\subset D$ is called a {\em subdimonoid} of a dimonoid $(D,\dashv, \vdash)$  if $a \dashv b, a \vdash b \in B$ for any $a, b \in B$. If the halo of $(D,\dashv, \vdash)$ is nonempty, then it is a subdimonoid of $(D,\dashv, \vdash)$.

An element $0\in D$  is called  a {\em   zero of a dimonoid} $(D,\dashv, \vdash)$~\cite{ZhADM2013nil} if $0$ is a zero of $(D,\dashv)$ and a  zero of $(D,\vdash)$. Let $(D,\dashv, \vdash)$ be a dimonoid and $0\notin D$. The binary operations defined on  $D$  can be extended to $D\cup\{0\}$ putting $0\dashv d=d\dashv 0=0=0\vdash d=d\vdash 0 $ for all $d\in D\cup \{0\}$. The notation $(D,\dashv, \vdash)^{+0}$  denotes a dimonoid $D\cup\{0\}$ obtained from $D$ by adjoining the extra zero $0$. 
It follows  that $\Halo((D,\dashv, \vdash)^{+0}) = \Halo(D,\dashv, \vdash)$.

\smallskip

A dimonoid $(D,\dashv, \vdash)$ is called {\em abelian}~\cite{ZhADM2015} if $x \dashv y = y \vdash x$ for all $x,y\in D$.

\smallskip

Let $(D,\dashv, \vdash)$ be a dimonoid. Define new operations $\dashv^d$ and  $\vdash^d$ on $D$ by 
$$x \dashv^d y = y \vdash x\ \ \text{  and  }\ \  x \vdash^d y = y \dashv x.$$
It is immediate to check that $(D,\dashv^d, \vdash^d)$ is a new dimonoid, called the {\em  dual dimonoid of $(D,\dashv, \vdash)$}~\cite{Lod}, which we denote by $(D,\dashv, \vdash)^d$. It follows that the unary duality operation $d: (D, \dashv, \vdash) \mapsto (D, \dashv, \vdash)^{d}$ is involutive in the sense that $((D,\dashv, \vdash)^d)^d=(D,\dashv, \vdash)$. In fact, $(D,\dashv, \vdash)^d$ is a dimonoid if and only if $(D,\dashv, \vdash)$ is a dimonoid.  As usual, a dimonoid $(D,\dashv, \vdash)$ is said to be {\em self-dual} if $(D,\dashv, \vdash)^d=(D,\dashv, \vdash)$. As established in~\cite{Gdim1}, a dimonoid $(D,\dashv, \vdash)$ is abelian if and only if it is self-dual, which in turn holds if and only if the semigroups $(D, \dashv)$ and $(D,\vdash)$ are dual to each other. Consequently, nonabelian dimonoids are divided into the pairs of dual dimonoids.

\smallskip

A dimonoid $(D,\dashv,\vdash)$ is called {\em commutative}~\cite{ZhADM2009} if both semigroups $(D,\dashv)$ and $(D,\vdash)$ are commutative.

\smallskip

Since commutative semigroups $(D,\dashv)$ and $(D,\vdash)$ are dual if and only if their operations coincide, all commutative nontrivial dimonoids are nonabelian, and all abelian nontrivial dimonoids are noncommutative. On the other hand,  all commutative trivial dimonoids are abelian and all noncommutative trivial dimonoids are nonabelian. A left zero and a right zero dimonoid $(D,\dashv,\vdash)$ with operations $x\dashv y = x$ and $x\vdash y = y$ \cite{Lod} is an example of a noncommutative abelian nontrivial dimonoid. In Section~\ref{sec:3el_dm}, we give examples of  nonabelian commutative dimonoids, see also~\cite{ZhADM2009}.

\bigskip

A map $\varphi : D_1 \to D_2$ is called a {\em homomorphism } from a dimonoid $(D_1,\dashv_1, \vdash_1)$ to a dimonoid $(D_2,\dashv_2, \vdash_2)$ if $$\varphi(a\dashv_1 b)=\varphi(a)\dashv_2\varphi(b)\ \ \text{  and  }\ \ \varphi(a\vdash_1 b)=\varphi(a)\vdash_2\varphi(b)$$ for all $a,b\in D_1$.

A bijective homomorphism $\psi : D_1 \to D_2$ is called an {\em isomorphism } from a dimonoid $(D_1,\dashv_1, \vdash_1)$ to a dimonoid $(D_2,\dashv_2, \vdash_2)$.
If there exists an isomorphism from a dimonoid $(D_1,\dashv_1, \vdash_1)$ to a dimonoid $(D_2,\dashv_2, \vdash_2)$, then $(D_1, \dashv_1, \vdash_1)$ and $(D_2, \dashv_2, \vdash_2)$ are said to be {\em isomorphic}, denoted $(D_1,\dashv_1, \vdash_1)\cong (D_2,\dashv_2, \vdash_2)$. An isomorphism $\psi: D\to D$ is called an {\em   automorphism} of a dimonoid $(D,\dashv, \vdash)$. By $\Aut(D,\dashv, \vdash)$ we denote the automorphism group of a dimonoid $(D,\dashv, \vdash)$. It follows that $\Aut((D,\dashv, \vdash)^{+0}) \cong \Aut(D,\dashv, \vdash)$.

\smallskip

The following proposition was proved in~\cite{Gdim1}.

\begin{proposition}\label{duality} Let $(D,\dashv, \vdash)$ be a dimonoid. Then $$\Aut(D,\dashv, \vdash)^d = \Aut(D,\dashv, \vdash)\ \text{ and }\ \Halo(D,\dashv, \vdash)^d = \Halo(D,\dashv, \vdash).$$ Moreover, $(D,\dashv, \vdash)$ is commutative if and only if $(D,\dashv, \vdash)^d$ is commutative.
\end{proposition}

\smallskip

The following proposition follows directly from the definitions.

\begin{proposition}\label{homo_dm_dual} Let $(D_1,\dashv_1, \vdash_1)$ and  $(D_2,\dashv_2, \vdash_2)$ be dimonoids. Then for a map $\varphi: D_1 \to D_2$ the following conditions are equivalent:
\begin{itemize}
\item[1)] $\varphi$ is a homomorphism from $(D_1,\dashv_1, \vdash_1)$ to $(D_2,\dashv_2, \vdash_2)$;
\item[2)] $\varphi$ is a homomorphism from $(D_1,\dashv_1, \vdash_1)^{d}$ to $(D_2,\dashv_2, \vdash_2)^{d}$.
\end{itemize}
\end{proposition}

\begin{corollary}\label{iso_dm_dual} Let $(D_1,\dashv_1, \vdash_1)$ and  $(D_2,\dashv_2, \vdash_2)$ be dimonoids. Then   $(D_1,\dashv_1, \vdash_1)$ and $(D_2,\dashv_2, \vdash_2)$ are isomorphic if and only if $(D_1,\dashv_1, \vdash_1)^{d}$ and  $(D_2,\dashv_2, \vdash_2)^{d}$ are isomorphic.
\end{corollary}

\medskip

For a dimonoid $(D,\dashv,\vdash)$, if $\mathbb S$ and $\mathbb T$ denote the semigroups $(D,\dashv)$ and $(D,\vdash)$, respectively, then $\mathbb S \rbag \mathbb T$ stands for the dimonoid $(D,\dashv,\vdash)$.

\smallskip

 A dimonoid $(D,\dashv, \vdash)$ is said to be {\em rectangular} if both $(D,\dashv)$ and $(D,\vdash)$ are rectangular semigroups. Well-known examples of rectangular dimonoids are dimonoids $LO_D\rbag RO_D$, $O_D\rbag RO_D$ and $LO_D\rbag O_D$, see~\cite{Lod, ZhADM2009, ZhADM2011a}.

\medskip

A $g$-dimonoid $(D, \dashv, \vdash)$ is said to be  {\em iso-dual} if it is isomorphic to its dual dimonoid $(D, \dashv, \vdash)^d$.

\medskip

Theorem~\ref{char_iso_com_rec_gdm} was proved in~\cite{Ggdim1}.
 
\begin{theorem}\label{char_iso_com_rec_gdm} Let $\kappa>1$ be a cardinal, and let $D$ be a set of cardinality $|D|=\kappa$ with distinct elements $0,z\in D$. Up to isomorphism, there exist exactly two commutative rectangular $g$-dimonoids of order $\kappa$: the trivial one $O_{D^0}$ with $\Aut(O_{D^0})\cong S_{D\setminus\{0\}}$, and the  nonabelian iso-dual nontrivial $g$-dimonoid $O_{D^0} \rbag O_{D^z}$ with $\Aut(O_{D^0} \rbag O_{D^z})\cong S_{D\setminus\{0,z\}}$, which is not a dimonoid.
\end{theorem}

Because the nonabelian iso-dual nontrivial $g$-dimonoid $O_{D^0} \rbag O_{D^z}$ is not a dimonoid, the following corollary is obtained.

\begin{corollary}\label{char_iso_com_rec_dm} Let $\kappa$ be a cardinal, and let $D$ be a set with $|D| = \kappa$.
Up to isomorphism, there exists exactly one commutative rectangular dimonoid of order $\kappa$, namely the trivial null dimonoid $O_{D^0}$, whose automorphism group is isomorphic to $S_{D \setminus \{0\}}$.
\end{corollary}

\section{Some classes of nonabelian noncommutative dimonoids}\label{sec:cnonabelian2}

In this section, we construct new classes of nonabelian noncommutative dimonoids and determine their halos and automorphism groups, which play an important role in the classification of dimonoids by capturing their symmetries and identifying structurally equivalent objects.

 It is evident that every abelian (self-dual)  dimonoid is iso-dual. Theorem~\ref{iso-dual_LO} below provides examples of iso-dual dimonoids that are not abelian.

\begin{theorem}\label{iso-dual_LO}
Let $D$ be a set, $A$ and $B$ be subsets of $D$ of cardinality $|A|=|B|\geq 2$, $c\in A\cap B$, and $a\in A$, $b\in B$ be different elements  such that $A\setminus\{a\} = B\setminus\{b\}$. An algebraic structure $LO_{A_c\leftarrow D} \rbag RO_{B_c\leftarrow D} = (D, \dashv, \vdash)$, where

\begin{center}
$x\dashv y=\begin{cases}
x,\text{ if } x\in A \\
c,\text{ if } x\in D\setminus A
\end{cases}$
and\ \ \ \  
$x\vdash y=\begin{cases}
y,\text{ if } y\in B \\
c,\text{ if } y\in D\setminus B
\end{cases}$
\end{center}

\noindent is an iso-dual nonabelian noncommutative rectangular dimonoid with empty halo and $\Aut(LO_{A_c\leftarrow D} \rbag RO_{B_c\leftarrow D})\cong S_{A\setminus\{a,c\}}\times S_{D\setminus(A\cup B)}$.

\end{theorem}

\begin{proof} Let us show that $(D, \dashv, \vdash)$ is a dimonoid. It follows directly from the definition of the dimonoid operations that, for all $x,y,z\in D$, 
$$x \dashv (y \dashv z) = x \dashv (y \vdash z)\ \text{ and }\ (x \dashv y) \vdash z = (x \vdash y) \vdash z.$$ 
Therefore, the axioms $(D_1)$ and $(D_3)$ are satisfied.

Let us verify the axiom $(D_2)$ by considering four separate cases. 

{\noindent \bf Case 1}. Let $y\in A\cap B$. Then 
$$(x \vdash y) \dashv z = y \dashv z = y = x \vdash y = x \vdash (y \dashv z)$$ 
for all $x,z\in D$.
Hence, in this case, the axiom $(D_2)$ holds.

{\noindent \bf Case 2}. Let $y\in A\setminus B=\{a\}$. It follows that 
$$(x \vdash a) \dashv z = c \dashv z = c = x \vdash a = x \vdash (a \dashv z)$$ 
for all $x,z\in D$. Therefore, in this case, the axiom $(D_2)$ is satisfied.

{\noindent \bf Case 3}. Let $y\in B\setminus A=\{b\}$. Then 
$$(x \vdash b) \dashv z = b \dashv z = c = x \vdash c = x \vdash (b \dashv z)$$ 
for all $x,z\in D$. Thus, in this case, the axiom $(D_2)$ holds.

{\noindent \bf Case 4}. Let $y\in D\setminus (A\cup B)$. It follows that 
$$(x \vdash y) \dashv z = c \dashv z = c = x \vdash c = x \vdash (y \dashv z)$$
 for all $x,z\in D$. We conclude that axiom $(D_2)$ is satisfied.

\smallskip

Since both semigroups $(D, \dashv)$ and $(D, \vdash)$ are rectangular, $(D, \dashv, \vdash)$ is a rectangular dimonoid.

\smallskip

Taking into account that $a\dashv b = a \ne c = b \vdash a$ and $a\dashv b = a \ne c = b \dashv a$, we conclude that $(D, \dashv, \vdash)$ is a nonabelian noncommutative dimonoid. 

\smallskip

Consider its dual dimonoid $LO_{B_c\leftarrow D} \rbag RO_{A_c\leftarrow D} = (D, \dashv^d, \vdash^d)$ with operations
\begin{center}
$x\dashv^d y=\begin{cases}
x,\text{ if } x\in B \\
c,\text{ if } x\in D\setminus B
\end{cases}$
and\ \ \ \  
$x\vdash^d y=\begin{cases}
y,\text{ if } y\in A \\
c,\text{ if } y\in D\setminus A.
\end{cases}$
\end{center}

One can check that the bijective map  $\psi:D\to D$ defined by $\psi(a)=b$, $\psi(b)=a$, and $\psi(x)=x$ for all $x\in D\setminus\{a,b\}$, is an isomorphism from $(D, \dashv, \vdash)$ to $(D, \dashv^d, \vdash^d)$.

\smallskip
 
Since the  semigroup $(D, \dashv)$ contains no right identities, the halo of the dimonoid $LO_{A_c\leftarrow D} \rbag RO_{B_c\leftarrow D}$ is empty.

\smallskip

Let $\psi$ be an arbitrary automorphism of the dimonoid $(D,\dashv, \vdash)$. Then $\psi$ is an automorphism of the semigroup $(D, \dashv)$ and $\psi$ is an automorphism of the semigroup $(D, \vdash)$. Since automorphisms preserve left zeros and $A$ is a subset of left zeros of $(D, \dashv)$, it follows that $\psi(A)=A$, $\psi(D\setminus A)= D\setminus A$, and hence $\psi(a)\in A$, $\psi(b)\in D\setminus A$. Taking into account that automorphisms preserve right zeros and $B$ is a subset of right zeros of $(D, \vdash)$, it follows that $\psi(B)=B$, $\psi(D\setminus B)= D\setminus B$, and hence $\psi(b)\in B$, $\psi(a)\in D\setminus B$. We conclude that $\psi(a)\in A\setminus B=\{a\}$ and $\psi(b)\in B\setminus A=\{b\}$, and thus $\psi(a) = a$ and $\psi(b) = b$. It follows also that $\psi(c)=\psi(c\vdash a)=\psi(c)\vdash \psi(a)= \psi(c)\vdash a = c$. Consequently, $\psi(A\setminus\{a,c\})= A\setminus\{a,c\}$ and $\psi(D\setminus (A\cup B))=D\setminus (A\cup B)$.

On the other hand, let $f$ be any bijection of $D$ such that $f(A\setminus\{a,c\})= A\setminus\{a,c\}$, $f(D\setminus (A\cup B))=D\setminus (A\cup B)$ and $f(a)= a$, $f(b)= b$, $f(c)=c$.  If $x\in A\setminus\{a,c\}$, then $f(x)\in A\setminus\{a,c\}$, and hence $f(x\dashv y)= f(x)= f(x)\dashv f(y)$ for all $y\in D$. By analogy, if $y\in A\setminus\{a,c\}= B\setminus\{b,c\}$, then $f(y)\in B\setminus\{b,c\}$, and hence $f(x\vdash y)= f(y)= f(x)\vdash f(y)$ for all $x\in D$.
In the case $x\in D\setminus (A\cup B)$ we have that $f(x)\in D\setminus (A\cup B)$, and thus $f(x\dashv y)= f(c)= c = f(x)\dashv f(y)$ for all $y\in D$. If $y\in D\setminus (A\cup B)$, then $f(y)\in D\setminus (A\cup B)$, and thus $f(x\vdash y)= f(c)= c = f(x)\vdash f(y)$ for all $x\in D$.
It follows that any bijection of $A\setminus\{a,c\}$ and any bijection of $D\setminus (A\cup B)$ generate an automorphism of the dimonoid $LO_{A_c\leftarrow D} \rbag RO_{B_c\leftarrow D}$. Therefore, $\Aut(LO_{A_c\leftarrow D} \rbag RO_{B_c\leftarrow D}) \cong S_{A\setminus\{a,c\}}\times S_{D\setminus(A\cup B)}$.
\end{proof}

In Theorem~\ref{LOB_RO_a_D} and Corollary~\ref{LO_a_D_ROB} we construct nonabelian dimonoids with nonempty halo.

\begin{theorem}\label{LOB_RO_a_D}
Let $D$ be a set of cardinality $|D|>1$,  $0\notin D$,  $D^{\sim 0}=D\cup\{0\}$, and $a\in D$, $c\in D^{\sim 0}$ be different elements. An algebraic structure $LOB_{D^{\sim 0}}^{a,c} \rbag RO^{\sim 0}_{\{a\}\leftarrow D} =(D^{\sim 0}, \dashv, \vdash)$, where

\begin{center}
$x\dashv y=\begin{cases}
a,\text{ if } x=y=a \\
c,\text{ if } x=a\text{ and } y\neq a\\
x,\text{ if } x\in D^{\sim 0}\setminus \{a\}
\end{cases}$
and\ \ \ \  
$x\vdash y=\begin{cases}
y,\text{ if } x=a \\
0,\text{ if } x\in D^{\sim 0}\setminus \{a\}
\end{cases}$
\end{center}

\noindent is a nonabelian noncommutative dimonoid with $\Halo(LOB_{D^{\sim 0}}^{a,c} \rbag RO^{\sim 0}_{\{a\}\leftarrow D}) = \{a\}$ and $\Aut(LOB_{D^{\sim 0}}^{a,c} \rbag RO^{\sim 0}_{\{a\}\leftarrow D}) \cong S_{D^{\sim 0}\setminus\{a,c,0\}}$.
\end{theorem}

\begin{proof} Let us show that $(D^{\sim 0}, \dashv, \vdash)$ is a dimonoid dividing into cases of axioms of a dimonoid. 

{\noindent \bf Case $(D_1)$}. If $x\ne a$, then $x$ is a left zero of a semigroup $(D^{\sim 0}, \dashv)$, and hence
$(x \dashv y) \dashv z = x \dashv (y \dashv z)= x = x \dashv (y \vdash z)$ for all $y,z\in D^{\sim 0}$.
If $x=a$, then $(x \dashv y) \dashv z = x \dashv (y \dashv z), x \dashv (y \vdash z)\in\{a,c\}$ for any $y,z\in D^{\sim 0}$. Since $(x \dashv y) \dashv z=a$ if and only if $x=y=z=a$, $x \dashv (y \vdash z)=a$ if and only if $x=y=z=a$, we conclude that $(D_1)$ holds.

{\noindent \bf Case $(D_2)$}. Since $a$ is a left identity of $(D^{\sim 0}, \vdash)$, it follows  $(a \vdash y) \dashv z = y \dashv z = a \vdash (y \dashv z)$ for all $y,z\in D^{\sim 0}$. If $x\ne a$, then $(x \vdash y) \dashv z = 0 \dashv z = 0 = x \vdash (y \dashv z)$ for all $y,z\in D^{\sim 0}$. We conclude that $(D_2)$ holds. 

{\noindent \bf Case $(D_3)$}. Taking into account that $a$ is a right identity of $(D^{\sim 0}, \dashv)$ and a left identity of $(D^{\sim 0}, \vdash)$, we conclude that
$(x \dashv a) \vdash z = x \vdash z = x \vdash (a \vdash z)$ for all $x,z\in D^{\sim 0}$. If $y\ne a$, then $x \dashv y \ne a$ and $y\vdash z = 0$ for all $x,z\in D^{\sim 0}$, and hence $(x \dashv y) \vdash z = 0 = x \vdash 0 = x \vdash (y \vdash z)$ for all $x,z\in D^{\sim 0}$. It follows that $(D_3)$ holds.

\smallskip

Taking into account that semigroups $(D^{\sim 0}, \dashv)$  and $(D^{\sim 0}, \vdash)$ are not dual, we conclude that $(D^{\sim 0}, \dashv, \vdash)$ is a nonabelian dimonoid. Since $(D^{\sim 0}, \dashv)$ is a noncommutative semigroup, $(D^{\sim 0}, \dashv, \vdash)$ is a noncommutative dimonoid.
 
Since  $a$ is a unique right identity of $(D^{\sim 0}, \dashv)$ and a unique left identity of $(D^{\sim 0}, \vdash)$, the halo of the dimonoid $LOB_{D^{\sim 0}}^{a,c} \rbag RO^{\sim 0}_{\{a\}\leftarrow D}$ is equal to $\{a\}$.

Let $\psi$ be an arbitrary automorphism of the dimonoid $(D^{\sim 0}, \dashv, \vdash)$. Then $\psi$ is an automorphism of the semigroup $(D^{\sim 0}, \dashv)$ and $\psi$ is an automorphism of the semigroup $(D^{\sim 0}, \vdash)$. Since automorphisms preserve zeros and $0$ is a zero of $(D^{\sim 0}, \vdash)$, we conclude that $\psi(0) = 0$. 
Since $a$ is a unique right identity of $(D^{\sim 0}, \dashv)$ and automorphisms preserve right identities, it follows that $\psi(a)= a$. Fix any $d\ne a$. It follows that $\psi(d)\ne a$, and hence $\psi(c)=\psi(a\dashv d) = \psi(a)\dashv \psi(d)= a\dashv \psi(d) = c$.

On the other hand, let $f$ be any bijection of $(D^{\sim 0}, \dashv, \vdash)$ such that $f(a)= a$, $f(c)=c$ and $f(0)=0$.  
If $x\ne a$, then $f(x)\ne a$, and hence $f(x\dashv y)= f(x)= f(x)\dashv f(y)$ and $f(x\vdash y)= f(0)= 0 = f(x)\vdash f(y)$ for all $y\in D^{\sim 0}$. If $y\ne a$, then $f(y)\ne a$, and thus $f(a\dashv y)= f(c)= c = a\dashv f(y) = f(a)\dashv f(y)$ for all $y\in D^{\sim 0}$. It follows also that $f(a\dashv a)= f(a)= a = a\dashv a = f(a)\dashv f(a)$. Moreover, $f(a\vdash y)= f(y)= a\vdash f(y) = f(a)\vdash f(y)$ for all $y\in D^{\sim 0}$.

Therefore, any bijection of $D^{\sim 0}$ that preserves $a$, $c$ and $0$ generates an automorphism of the dimonoid $LOB_{D^{\sim 0}}^{a,c} \rbag RO^{\sim 0}_{\{a\}\leftarrow D}$. Consequently, $\Aut(LOB_{D^{\sim 0}}^{a,c} \rbag RO^{\sim 0}_{\{a\}\leftarrow D}) \cong S_{D^{\sim 0}\setminus\{a,c,0\}}$.
\end{proof}

Theorem~\ref{LOB_RO_a_D} and Proposition~\ref{duality} yield the following corollary.

\begin{corollary}\label{LO_a_D_ROB}
Let $D$ be a set of cardinality $|D|>1$,  $0\notin D$,  $D^{\sim 0}=D\cup\{0\}$, and $a\in D$, $c\in D^{\sim 0}$ be different elements. An algebraic structure $LO^{\sim 0}_{\{a\}\leftarrow D} \rbag ROB_{D^{\sim 0}}^{a,c} =(D^{\sim 0}, \dashv, \vdash)$, where

\begin{center}
$x\dashv y=\begin{cases}
x,\text{ if } y=a \\
0,\text{ if } y\in D^{\sim 0}\setminus \{a\}
\end{cases}$
and\ \ \ \
$x\vdash y=\begin{cases}
a,\text{ if } x=y=a \\
c,\text{ if } x\neq a\text{ and }y=a \\
y,\text{ if } y\in D^{\sim 0}\setminus \{a\}
\end{cases}$

\end{center}

\noindent is a nonabelian noncommutative dimonoid with $\Halo(LO^{\sim 0}_{\{a\}\leftarrow D} \rbag ROB_{D^{\sim 0}}^{a,c}) = \{a\}$ and $\Aut(LO^{\sim 0}_{\{a\}\leftarrow D} \rbag ROB_{D^{\sim 0}}^{a,c})\cong S_{D^{\sim 0}\setminus\{a,c,0\}}$.
\end{corollary}

\begin{proposition}\label{iso_LOB_RO} Let $D$ be a set of cardinality $|D|>1$,  $0\notin D$,  $D^{\sim 0}=D\cup\{0\}$, and $a\in D$, $c\in D^{\sim 0}$ be different elements. If $\{0,b\}\subset D^{\sim 0}\setminus\{a,c\}$, then the dimonoids  $LOB_{D^{\sim 0}}^{a,c} \rbag RO^{\sim b}_{\{a\}\leftarrow D^{\sim 0}\setminus\{b\}}$ and $LOB_{D^{\sim 0}}^{a,c} \rbag RO^{\sim 0}_{\{a\}\leftarrow D^{\sim 0}\setminus\{0\}}$ are isomorphic while the dimonoids  $LOB_{D^{\sim 0}}^{a,c} \rbag RO^{\sim b}_{\{a\}\leftarrow D^{\sim 0}\setminus\{b\}}$ and $LOB_{D^{\sim 0}}^{a,c} \rbag RO^{\sim c}_{\{a\}\leftarrow D^{\sim 0}\setminus\{c\}}$ are not isomorphic, where the semigroups $LOB_{D^{\sim 0}}^{a,c}$ and $RO^{\sim p}_{\{a\}\leftarrow D^{\sim 0}\setminus\{p\}}$ for $p\in\{0,b,c\}$ are endowed with the operations

\begin{center}
$x\dashv y=\begin{cases}
a,\text{ if } x=y=a \\
c,\text{ if } x=a\text{ and } y\neq a\\
x,\text{ if } x\in D^{\sim 0}\setminus \{a\}
\end{cases}$
and\ \ \ \  
$x\vdash_p y=\begin{cases}
y,\text{ if } x=a \\
p,\text{ if } x\in D^{\sim 0}\setminus \{a\}.
\end{cases}$
\end{center}
\end{proposition}

\begin{proof}One may verify that the dimonoids  $LOB_{D^{\sim 0}}^{a,c} \rbag RO^{\sim b}_{\{a\}\leftarrow D^{\sim 0}\setminus\{b\}}$ and $LOB_{D^{\sim 0}}^{a,c} \rbag RO^{\sim 0}_{\{a\}\leftarrow D^{\sim 0}\setminus\{0\}}$ are isomorphic via the bijection $\psi: D^{\sim 0} \to D^{\sim 0}$ defined by $\psi(0) = b$, $\psi(b) = 0$, and $\psi(x) = x$ for all $x\notin\{0,b\}$. 

To prove that the dimonoids  $LOB_{D^{\sim 0}}^{a,c} \rbag RO^{\sim b}_{\{a\}\leftarrow D^{\sim 0}\setminus\{b\}}$ and $LOB_{D^{\sim 0}}^{a,c} \rbag RO^{\sim c}_{\{a\}\leftarrow D^{\sim 0}\setminus\{c\}}$ are not isomorphic, suppose, to the contrary, that $\phi:LOB_{D^{\sim 0}}^{a,c} \rbag RO^{\sim c}_{\{a\}\leftarrow D^{\sim 0}\setminus\{c\}} \to LOB_{D^{\sim 0}}^{a,c} \rbag RO^{\sim b}_{\{a\}\leftarrow D^{\sim 0}\setminus\{b\}}$ is an isomorphism. Then $\phi$ is an automorphism of the semigroup $LOB_{D^{\sim 0}}^{a,c}$ and $\phi$ is an isomorphism from the semigroup $RO^{\sim c}_{\{a\}\leftarrow D^{\sim 0}\setminus\{c\}}$ to the semigroup $RO^{\sim b}_{\{a\}\leftarrow D^{\sim 0}\setminus\{b\}}$. Since isomorphisms preserve zeros, and $c$ and $b$ are zeros of $RO^{\sim c}_{\{a\}\leftarrow D^{\sim 0}\setminus\{c\}}$ and $RO^{\sim b}_{\{a\}\leftarrow D^{\sim 0}\setminus\{b\}}$, respectively, we obtain $\phi(c) = b$. On the other hand, $\phi(c) = c$ by the proof of Theorem~\ref{LOB_RO_a_D}, which yields a contradiction. 
\end{proof}

From Corollary~\ref{iso_dm_dual} and Proposition~\ref{iso_LOB_RO} we obtain the following proposition.

\begin{proposition} Let $D$ be a set of cardinality $|D|>1$,  $0\notin D$,  $D^{\sim 0}=D\cup\{0\}$, and $a\in D$, $c\in D^{\sim 0}$ be different elements. If $\{0,b\}\subset D^{\sim 0}\setminus\{a,c\}$, then the dimonoids  $LO^{\sim b}_{\{a\}\leftarrow D^{\sim 0}\setminus\{b\}} \rbag ROB_{D^{\sim 0}}^{a,c}$ and $LO^{\sim 0}_{\{a\}\leftarrow D^{\sim 0}\setminus\{0\}} \rbag ROB_{D^{\sim 0}}^{a,c}$ are isomorphic while the dimonoids  $LO^{\sim b}_{\{a\}\leftarrow D^{\sim 0}\setminus\{b\}} \rbag ROB_{D^{\sim 0}}^{a,c} =(D^{\sim 0}, \dashv, \vdash)$ and $LO^{\sim c}_{\{a\}\leftarrow D^{\sim 0}\setminus\{c\}} \rbag ROB_{D^{\sim 0}}^{a,c}$ are not isomorphic, where the semigroups  $LO^{\sim p}_{\{a\}\leftarrow D^{\sim 0}\setminus\{p\}}$ for $p\in\{0,b,c\}$ and $ROB_{D^{\sim 0}}^{a,c}$ are endowed with the operations

\begin{center}
$x\dashv_p y=\begin{cases}
x,\text{ if } y=a \\
p,\text{ if } y\in D^{\sim 0}\setminus \{a\}
\end{cases}$
and\ \ \ \ 
$x\vdash y=\begin{cases}
a,\text{ if } x=y=a \\
c,\text{ if } x\neq a\text{ and }y=a \\
y,\text{ if } y\in D^{\sim 0}\setminus \{a\}.
\end{cases}$
\end{center}
\end{proposition}

\bigskip

The following proposition was proved in~\cite{Gdim1}.

\begin{proposition}\label{lodim}
Let $(D, \dashv)$ be a semigroup  and $(D,\vdash)$ be a null semigroup with zero $0$. An algebraic structure $(D,\dashv, \vdash)$ is a dimonoid if and only if $0$ is a left zero of $(D, \dashv)$ and $x\dashv y\dashv z = x\dashv 0$ for all $x,y,z\in D$.
\end{proposition}

\begin{theorem}\label{LOA_O} Let $A$ be a proper subset of a set $D$ such that $a\in A$ and $0\in A$. Then $LO_{A_a^0\leftarrow D}\rbag  O_{D^0} = (D,\dashv, \vdash)$, where $(D,\vdash)$ is a null semigroup with zero $0$ and
\begin{center}
$x\dashv y=\begin{cases}
x,\text{ if } x\in A \\
a,\text{ if } x \in D\setminus A,
\end{cases}$
\end{center}
is a rectangular dimonoid with empty halo and $\Aut(LO_{A_a^0\leftarrow D}\rbag  O_{D^0}) \cong S_{A\setminus\{0, a\}}\times S_{D\setminus A}$. Moreover, if $|A|>1$, then  $LO_{A_a^0\leftarrow D}\rbag  O_{D^0}$ is a nonabelian noncommutative dimonoid. 
\end{theorem}

\begin{proof}
Using the definition of the semigroup $LO_{A_a^0\leftarrow D}$, it is immediate to check that $x\dashv y\dashv z = x\dashv 0$ for all $x,y,z\in D$. Since $0$ is a left zero of $LO_{A_a^0\leftarrow D}$, Proposition~\ref{lodim} implies that $LO_{A_a^0\leftarrow D}\rbag  O_{D^0}$ is a dimonoid. Taking into account that $O_{D^0}$ and $LO_{A_a^0\leftarrow D}$ are rectangular semigroups, we conclude that the dimonoid $LO_{A_a^0\leftarrow D}\rbag O_{D^0}$ is rectangular as well. 

Since in the case $|D|> 1$ the commutative semigroup $O_{D^0}$  contains no identity, $\Halo(LO_{A_a^0\leftarrow D}\rbag O_{D^0}) = \emptyset$.

Let $\psi$ be an arbitrary automorphism of the dimonoid $(D,\dashv, \vdash)$. Then $\psi$ is an automorphism of the semigroup $(D, \dashv)$ and $\psi$ is an automorphism of the semigroup $(D, \vdash)$. Since automorphisms preserve zeros and $0$ is a zero of $(D, \vdash)$, we conclude that $\psi(0) = 0$.
Taking into account that automorphisms preserve left zeros and $A$ is a left zero subsemigroup of $(D, \dashv)$, we conclude that $\psi(A)=A$, and thus $\psi(D\setminus A)=D\setminus A$. Fix any $d\in D\setminus A$. It follows that $\psi(d)\in D\setminus A$, and hence $\psi(a)=\psi(d\dashv a) = \psi(d)\dashv \psi(a)=a$.

On the other hand, let $f$ be any bijection of $D$ such that $f(A)=A$, $f(a)=a$, and $f(0)=0$.  If $x\in A$, then $f(x)\in A$, and hence $f(x\dashv y)= f(x)= f(x)\dashv f(y)$ for all $y\in D$. In the case $x\in D\setminus A$ we have that $f(x)\in D\setminus A$, and thus $f(x\dashv y)= f(a)= a = f(x)\dashv f(y)$ for all $y\in D$. Moreover, $f(x\vdash y)= f(0)= 0 = f(x)\vdash f(y)$ for all $x, y\in D$. It follows that any bijection of $A$ that preserves $a$ and $0$ and any bijection of $D\setminus A$ generate an automorphism of $LO_{A_a^0\leftarrow D}\rbag  O_{D^0}$. Therefore, $\Aut(LO_{A_a^0\leftarrow D}\rbag  O_{D^0}) \cong S_{A\setminus\{0, a\}}\times S_{D\setminus A}$.

Since in the case $|A|>1$ the semigroup $LO_{A_a^0\leftarrow D}$ contains at least two left zeros, and hence it is not commutative, we conclude that the dimonoid $LO_{A_a^0\leftarrow D} \rbag  O_{D^0}$ is not commutative as well.   Taking into account that in the case $|A|>1$ the semigroup $LO_{A_a^0\leftarrow D}$ is not commutative while the semigroup $O_{D^0}$ is commutative, we conclude that the semigroups   $LO_{A_a^0\leftarrow D}$ and $O_{D^0}$ cannot be dual, and thus  the dimonoid $LO_{A_a^0\leftarrow D}\rbag  O_{D^0}$ is not abelian.
\end{proof}

Theorem~\ref{LOA_O} and Proposition~\ref{duality} imply the following corollary.

\begin{corollary} Let $A$ be a proper subset of a set $D$ such that $a\in A$ and $0\in A$. Then $O_{D^0} \rbag RO_{A_a^0\leftarrow D}  = (D,\dashv, \vdash)$, where $(D,\dashv)$ is a null semigroup with zero $0$ and
\begin{center}
$x\vdash y=\begin{cases}
y,\text{ if } y\in A \\
a,\text{ if } y \in D\setminus A,
\end{cases}$
\end{center}
is a rectangular dimonoid with empty halo and $\Aut(O_{D^0} \rbag RO_{A_a^0\leftarrow D}) \cong S_{A\setminus\{0, a\}}\times S_{D\setminus A}$. Moreover, if $|A|>1$, then  $O_{D^0} \rbag RO_{A_a^0\leftarrow D}$ is a nonabelian noncommutative dimonoid. 
\end{corollary}

\begin{proposition}\label{LO_O_iso} Let $A$ be a proper subset of a set $D$ such that $0\in A$. If $\{a,b\}\subset A\setminus\{0\}$, then the dimonoids  $LO_{A_a^0\leftarrow D}\rbag O_{D^0}$ and $LO_{A_b^0\leftarrow D}\rbag O_{D^0}$ are isomorphic while the dimonoids  $LO_{A_a^0\leftarrow D}\rbag O_{D^0}$ and $LO_{A_0\leftarrow D}\rbag O_{D^0}$ are not isomorphic, where $O_{D^0}$ is a null semigroup with zero $0$ and for $p\in\{a,b,0\}$ the semigroup $LO_{A_p^0\leftarrow D}$ is endowed with the operation
\begin{center}
$x\dashv_p y=\begin{cases}
x,\text{ if } x\in A \\
p,\text{ if } x \in D\setminus A.
\end{cases}$
\end{center}
\end{proposition}

\begin{proof}One can verify that the dimonoids  $LO_{A_a^0\leftarrow D}\rbag O_{D^0}$ and $LO_{A_b^0\leftarrow D}\rbag O_{D^0}$ are isomorphic via the bijection $\psi: D \to D$ defined by $\psi(a) = b$, $\psi(b) = a$, and $\psi(x) = x$ for all $x\notin\{a,b\}$. 

To prove that the dimonoids  $LO_{A_a^0\leftarrow D}\rbag O_{D^0}$ and $LO_{A_0\leftarrow D}\rbag O_{D^0}$ are not isomorphic, suppose, to the contrary, that $\phi:LO_{A_0\leftarrow D}\rbag O_{D^0} \to LO_{A_a^0\leftarrow D}\rbag O_{D^0}$ is an isomorphism. Then $\phi$ is an isomorphism from the semigroup $LO_{A_0\leftarrow D}$ to the semigroup $LO_{A_a^0\leftarrow D}$ and $\phi$ is an automorphism of the semigroup $O_{D^0}$. Taking into account that isomorphisms preserve left zeros and $A$ is a left zero subsemigroup of the semigroups $LO_{A_0\leftarrow D}$ and $LO_{A_a^0\leftarrow D}$, we conclude that $\phi(A)=A$, and thus $\phi(D\setminus A)=D\setminus A$. Fix any $d\in D\setminus A$. It follows that $\phi(d)\in D\setminus A$, and hence $\phi(0)=\phi(d\dashv_0 0) = \phi(d)\dashv_a \phi(0)= a$. Since automorphisms preserve zeros and $0$ is a zero of $O_{D^0}$, we obtain $\phi(0) = 0$, which yields a contradiction. 
\end{proof}

From Proposition~\ref{LO_O_iso} and Corollary~\ref{iso_dm_dual} we obtain the following proposition.

\begin{proposition} Let $A$ be a proper subset of a set $D$ such that $0\in A$. If $\{a,b\}\subset A\setminus\{0\}$, then the dimonoids  $O_{D^0} \rbag  RO_{A_a^0\leftarrow D}$ and $O_{D^0} \rbag  RO_{A_b^0\leftarrow D}$ are isomorphic while the dimonoids  $O_{D^0} \rbag  RO_{A_a^0\leftarrow D}$ and $O_{D^0} \rbag  RO_{A_0\leftarrow D}$ are not isomorphic, where $O_{D^0}$ is a null semigroup with zero $0$ and for $p\in\{a,b,0\}$ the semigroup $RO_{A_p^0\leftarrow D}$ is endowed with the operation
\begin{center}
$x\vdash_p y=\begin{cases}
y,\text{ if } y\in A \\
p,\text{ if } y \in D\setminus A.
\end{cases}$
\end{center}
\end{proposition}

\section{Three-element dimonoids}\label{sec:3el_dm}

In this section, we focus on describing, up to isomorphism, all dimonoids of order $3$.

Among the $19683$ possible binary operations on a three-element set $S$, precisely $113$ are associative. In other words, there exist exactly $113$ distinct three-element semigroups. However, many of these semigroups are isomorphic, and as a result, there are essentially only $24$ pairwise nonisomorphic semigroups of order $3$, see \cite{Ch, G8, G9}.

Among these $24$ pairwise nonisomorphic semigroups of order $3$, there are $12$ commutative semigroups. The remaining $12$ pairwise nonisomorphic noncommutative semigroups are partitioned into pairs of dual semigroups. Moreover, the automorphism groups of dual semigroups coincide.

Lists of all pairwise nonisomorphic semigroups of order $3$ and their automorphism groups are presented in Table~\ref{tab:auts3} and Table~\ref{tab:autns3} taken from~\cite{G9}. In these tables, the notation $S^{+1}$ denotes a monoid obtained from $S$ by adjoining the extra identity $1$ (whether or not $S$ is a monoid), and $S^{+0}$ denotes a semigroup obtained from $S$ by adjoining the extra zero $0$ (regardless of whether $S$ has a zero). The notation $M^{\tilde{1}}$ stands for the semigroup obtained from a monoid $M$ with  identity $1$ by adjoining an extra element $\tilde{1}$ by putting $\tilde{1}* m=m* \tilde{1}=m$ for all $m\in M$ and ${\tilde{1}}*{\tilde{1}}=1$. The notation $\M_{r,m}$ refers to a finite monogenic semigroup of index $r$ and period $m$, $L_n$ denotes the  linear semilattice $\{0, 1,\ldots, n\!-\!1\}$ of order $n$ endowed with the operation of minimum, and $S_n$ stands for the group of all permutations of $X$ with $|X|=n$. More details about these semigroups can be found in \cite{GR1}.

\begin{table}[H]
\centering
\resizebox{11cm}{!}{
\begin{tabular}{|c|c|c|c|c|c|c|c|c|c|c|c|c|}
        \hline
        $S$ & $C_3$ & $O_3$  & $\M_{2,2}$ & $C_2^{+1}$ &  $C_2^{\tilde{1}}$ & $\M_{3,1}$&  $O_2^{+1}$ & $O_2^{+0}$ & $L_3$ & $C_2^{+0}$    & $O_3^2$  & $O_3^1$ \\
        \hline
        $\Aut(S)$ &  $C_2$ & $C_2$ & $C_1$ & $C_1$ & $C_{1}$  & $C_{1}$ & $C_{1}$ & $C_{1}$ & $C_{1}$ & $C_{1}$ & $C_{2}$ & $C_{1}$  \\
        \hline
\end{tabular}
}
\smallskip
\caption{Nonisomorphic commutative $3$-element semigroups and their automorphism groups}\label{tab:auts3}
\end{table}

\begin{table}[H]
\centering
\resizebox{14cm}{!}{
\begin{tabular}{|c|c|c|c|c|c|c|}
            \hline
            $S$ & $LO_3$, $RO_3$  & $LO_2^{+0}$, $RO_2^{+0}$ &  $LO^{\sim 0}_{1\leftarrow2}$, $RO^{\sim 0}_{1\leftarrow2}$ & $LO_2^{+1}$, $RO_2^{+1}$ & $LOB_3$, $ROB_3$  & $LO_{2\leftarrow 3}$, $RO_{2\leftarrow 3}$ \\
            \hline
            $\Aut(S)$ & $S_3$ & $C_2$ & $C_1$ & $C_{2}$ & $C_{1}$ & $C_{2}$ \\
            \hline
\end{tabular}
}
\smallskip
\caption{Nonisomorphic noncommutative $3$-element semigroups and their automorphism groups}\label{tab:autns3}
\end{table}

The following theorem,  established in~\cite{Gdim2}, provides a complete classification of all pairwise nonisomorphic commutative dimonoids of order~$3$.

\begin{theorem}
Up to isomorphism, there exist $14$  three-element commutative dimonoids among which $12$ trivial dimonoids and a pair of  nonabelian  nontrivial dual dimonoids. 
\end{theorem}

\bigskip

Table~\ref{tab:commdim3},  taken from~\cite{Gdim2}, lists all pairwise nonisomorphic commutative nontrivial three-element dimonoids and their corresponding automorphism groups.

\begin{table}[H]
\centering
\resizebox{4.5cm}{!}{
    \begin{tabular}{|c|c|c|}
        \hline
        $D$ &  $\M_{3,1} \rbag O_3$ & $O_3 \rbag \M_{3,1}$\\
        \hline
        $\Aut(D)$ & $C_{1}$ & $C_{1}$ \\
        \hline
    \end{tabular}
}
\smallskip
\caption{Nonisomorphic commutative nontrivial three-element dimonoids and their automorphism groups}\label{tab:commdim3}
\end{table}

The following theorem,  proved in~\cite{Gdim2}, provides a complete classification of all pairwise nonisomorphic abelian dimonoids of order~$3$.

\begin{theorem}
Up to isomorphism, there exist $17$  three-element abelian dimonoids among which $12$ commutative trivial dimonoids and $5$ noncommutative nontrivial dimonoids.
\end{theorem}

The following Table~\ref{tab:noncommabdim3},  taken from~\cite{Gdim2}, lists all pairwise nonisomorphic  abelian  nontrivial dimonoids of order  $3$ and their corresponding automorphism groups.

\begin{table}[H]
\centering
\resizebox{12.5cm}{!}{
\begin{tabular}{|c|c|c|c|c|c|c|}
            \hline
            $D$  & $LO_3\rbag RO_3$  & $LO_{2\leftarrow 3}\rbag RO_{2\leftarrow 3}$ & $LOB_3\rbag ROB_3$ &   $LO^{\sim 0}_{1\leftarrow 2}\rbag RO^{\sim 0}_{1\leftarrow 2}$ & $(LO_2\rbag RO_2)^{+0}$  \\
             \cline{1-6}
            $\Aut(D)$ & $S_3$ & $C_2$ & $C_1$ & $C_{1}$ & $C_{2}$  \\
            \cline{1-6}
\end{tabular}

}
\smallskip
\caption{Nonisomorphic abelian nontrivial $3$-element dimonoids and their automorphism groups}\label{tab:noncommabdim3}
\end{table}

Based on the results of~\cite{Gdim1} and results of Section~\ref{sec:cnonabelian2} concerning noncommutative nonabelian dimonoids and their automorphism groups and properties of dual dimonoids, Table~\ref{tab:noncommdim3} lists pairwise nonisomorphic noncommutative nonabelian nontrivial three-element dimonoids and their automorphism groups.

\begin{table}[H]
\raggedright
\hspace{1cm}
\resizebox{15cm}{!}{
\begin{tabular}{|c|c|c|c|c|c|c|c|}
    \hline
    $D$ & $(LO_2\rbag O_2)^{+0}$   & $LO_3\rbag O_3$   & $LO_3\rbag RO_{2\leftarrow 3}$ & $LO_3\rbag LO_{2\leftarrow 3}$ &  $LO_{2_0\leftarrow 3}\rbag O_3$  &  $LO_{2_a^0\leftarrow 3}\rbag O_3$ &   $LO^{\sim 0}_{1\leftarrow2}\rbag O_3^1$ \\
    \hline
    $\Aut(D)$ & $C_1$  & $C_2$   & $C_{2}$ & $C_{2}$ & $C_1$ &  $C_1$ &  $C_{1}$\\
    \hline
    \hline
    $D$ & $(O_2\rbag RO_2)^{+0}$   & $O_3\rbag RO_3$   &  $LO_{2\leftarrow 3}\rbag RO_3$ &  $RO_{2\leftarrow 3}\rbag RO_3$ & $O_3\rbag RO_{2_0\leftarrow 3}$ & $O_3\rbag RO_{2_a^0\leftarrow 3}$ &  $O_3^1 \rbag RO^{\sim 0}_{1\leftarrow2}$\\
    \hline
    $\Aut(D)$ & $C_{1}$  & $C_2$ &  $C_{2}$ & $C_{2}$ &  $C_1$ & $C_1$ &  $C_{1}$\\
    \hline
\end{tabular}
}

\hspace{1cm}     
\resizebox{15cm}{!}{
\begin{tabular}{|c|c|c|c|c|}
    \hline
    $D$   & $LOB_3\rbag O_3^1$  &  $LOB_{\{a,c,0\}}^{a,c} \rbag RO^{\sim 0}_{\{a\}\leftarrow \{a,b\}}$  & $LOB_{\{a,c,0\}}^{a,c} \rbag RO^{\sim c}_{\{a\}\leftarrow \{a,0\}}$ & $LO_{\{a,c\}_c\leftarrow \{a,b,c\}} \rbag RO_{\{b,c\}_c\leftarrow \{a,b,c\}}$ \\
    \hline
    $\Aut(D)$  & $C_1$ & $C_{1}$ & $C_{1}$ & $C_{1}$ \\
    \hline
\end{tabular}
}

\hspace{1cm}     
\resizebox{10.6cm}{!}{
\begin{tabular}{|c|c|c|c|}
    \hline
    $D$    & $O_3^1 \rbag ROB_3$ &   $LO^{\sim 0}_{\{a\}\leftarrow \{a,b\}} \rbag ROB_{\{a,c,0\}}^{a,c}$  & $LO^{\sim c}_{\{a\}\leftarrow \{a,0\}} \rbag ROB_{\{a,c,0\}}^{a,c}$    \\
    \hline
    $\Aut(D)$   & $C_{1}$ & $C_{1}$ & $C_{1}$  \\
    \hline
\end{tabular}
}
\smallskip
\caption{Nonisomorphic nonabelian  noncommutative nontrivial $3$-element dimonoids  and their automorphism groups}\label{tab:noncommdim3}
\end{table}


By performing computer calculations, we found that, up to isomorphism, there exist exactly $21$ nonabelian noncommutative nontrivial dimonoids of order $3$.   Accordingly, we establish Theorem~\ref{niso_na_nc}, which resolves Problem 4.6 posed in~\cite{Gdim2}. 

\begin{theorem}\label{niso_na_nc}
Up to isomorphism, there exist $21$ nonabelian noncommutative nontrivial dimonoids of order $3$ among which there are exactly  $10$ pairs of  dual dimonoids, and a single iso-dual dimonoid.
\end{theorem}

In view of the foregoing, Theorem~\ref{comp_dm3} provides a complete classification, up to isomorphism, of all dimonoids of order 3.

\begin{theorem}\label{comp_dm3}
Up to isomorphism, there exist $52$ dimonoids of order $3$, among which $5$  abelian nontrivial dimonoids and a pair of dual  commutative nontrivial dimonoids. Furthermore, there are $21$ nonabelian noncommutative nontrivial dimonoids of order $3$, comprising exactly $10$ pairs of dual dimonoids and a single iso-dual dimonoid. In addition, there are  $12$ commutative (abelian) trivial dimonoids and $6$ pairs of noncommutative (nonabelian) dual trivial dimonoids.
\end{theorem}

\section{Number of dimonoids of small order}\label{sec:abelian}

The determination of the numbers $\mathrm{s}(n)$ and $\mathrm{cs}(n)$, representing respectively the counts of all pairwise nonisomorphic semigroups of order 
$n$ and all pairwise nonisomorphic commutative semigroups of order $n$, is a difficult combinatorial problem. The functions $\mathrm{s}(n)$ and $\mathrm{cs}(n)$
grow very rapidly as $n$ tends to infinity. The sequences $(\mathrm{s}(n))$ and $(\mathrm{cs}(n))$ are listed in the On-Line Encyclopedia of Integer Sequences as entries A027851 and A001426, respectively. All currently known exact values of these sequences are presented in Tables~\ref{tab:sg} and~\ref{tab:csg}.

\medskip

\begin{table}[H]
\centering
\resizebox{13cm}{!}{
\begin{tabular}{|r|c|c|c|c|c|c|c|c|c|c|c|}
            \hline
            $n$\ \  & 1  & 2 & 3 &   4 & 5  & 6 & 7 & 8 & 9  \\
            \cline{1-10}
            $\mathrm{s}(n)$  & 1 & 5 & 24 & 188 & 1915 & 28634 & 1627672 & 3684030417 & 105978177936292  \\
            \cline{1-10}
\end{tabular}
}
\smallskip
\caption{Number of nonisomorphic semigroups up to order $9$}\label{tab:sg}
\end{table}

\begin{table}[H]
\centering
\resizebox{12.5cm}{!}{
\begin{tabular}{|r|c|c|c|c|c|c|c|c|c|c|c|}
            \hline
            $n$\ \  & 1  & 2 & 3 &   4 & 5  & 6 & 7 & 8 & 9 & 10 \\
            \cline{1-11}
            $\mathrm{cs}(n)$  & 1 & 3 & 12 & 58 & 325 & 2143 & 17291 & 221805 & 11545843 & 3518930337 \\
            \cline{1-11}
\end{tabular}
}
\smallskip
\caption{Number of nonisomorphic commutative semigroups up to order $10$}\label{tab:csg}
\end{table}

Denote by $\mathrm{dm}(n)$, $\mathrm{cdm}(n)$, $\mathrm{adm}(n)$, and $\mathrm{rdm}(n)$ the number of all pairwise nonisomorphic dimonoids, commutative dimonoids, abelian dimonoids, and rectangular dimonoids of order $n$, respectively. We were able to calculate these cardinalities for small $n$. 
In Appendix~\ref{appnd}, we explain the method used to generate all pairwise nonisomorphic dimonoids of order $n$ and provide a listing of the \texttt{Python} code employed for these computations. The results of (computer) calculations are presented in Tables~\ref{tab:dm}--\ref{tab:rdm}.

\begin{table}[H]
\centering
\resizebox{6cm}{!}{
        \begin{tabular}{|r|c|c|c|c|c|c|}
            \hline
            $n$\ \  & 1  & 2 & 3 &   4 & 5   \\
            \cline{1-6}
            $\mathrm{dm}(n)$  & 1 & 8 & 52 & 734 & 55883  \\
            \cline{1-6}
         \end{tabular}
}
\smallskip
\caption{Number of  nonisomorphic dimonoids up to order 5}\label{tab:dm}
\end{table}

\begin{table}[H]
\centering
\resizebox{7cm}{!}{
        \begin{tabular}{|r|c|c|c|c|c|c|c|}
            \hline
            $n$\ \  & 1  & 2 & 3 &   4 & 5 & 6  \\
            \cline{1-7}
            $\mathrm{cdm}(n)$  & 1 & 3 & 14 & 101 & 1495 & 102268 \\
            \cline{1-7}
        \end{tabular}
}
\smallskip
\caption{ Number of  nonisomorphic commutative dimonoids up to order 6}\label{tab:cdm}
\end{table}

Since a dimonoid is considered trivial when its two operations coincide, the numbers of all pairwise nonisomorphic nontrivial dimonoids and pairwise nonisomorphic commutative nontrivial dimonoids of order $n$ are equal to $\mathrm{dm}(n)\!-\!\mathrm{s}(n)$ and  $\mathrm{cdm}(n)\!-\!\mathrm{cs}(n)$, respectively.

\begin{table}[H]
\centering
\resizebox{6.5cm}{!}{
        \begin{tabular}{|r|c|c|c|c|c|c|c|}
            \hline
            $n$\ \  & 1  & 2 & 3 &   4 & 5 & 6  \\
            \cline{1-7}
            $\mathrm{adm}(n)$  & 1 & 4 & 17 & 103 & 791 & 10870 \\
            \cline{1-7}
        \end{tabular}
}
\smallskip
\caption{ Number of  nonisomorphic abelian dimonoids up to order 6}\label{tab:adm}
\end{table}

Since a commutative dimonoid is abelian if and only if it is trivial, the number of all pairwise nonisomorphic abelian trivial dimonoids of order $n$ is equal to $\mathrm{cs}(n)$, and hence the numbers of all pairwise nonisomorphic nonabelian commutative dimonoids and pairwise nonisomorphic noncommutative abelian dimonoids of order~$n$ are $\mathrm{cdm}(n)\!-\!\mathrm{cs}(n)$ and $\mathrm{adm}(n)\!-\!\mathrm{cs}(n)$, respectively.

\begin{table}[H]
\centering
\resizebox{6cm}{!}{
        \begin{tabular}{|r|c|c|c|c|c|c|}
            \hline
            $n$\ \  & 1  & 2 & 3 &   4 & 5 & 6  \\
            \cline{1-7}
            $\mathrm{rdm}(n)$  & 1 & 6 & 18 & 62 & 166 & 509 \\
            \cline{1-7}
            $\mathrm{rs}(n)$  & 1 & 3 & 5 & 10 & 14 & 27 \\
            \cline{1-7}
        \end{tabular}
}
\smallskip
\caption{ Number of  nonisomorphic rectangular dimonoids and semigroups up to order 6}\label{tab:rdm}
\end{table}

Table~\ref{tab:rdm} also lists the numbers $\mathrm{rs}(n)$ of all pairwise nonisomorphic rectangular semigroups for $n \leq 6$. Hence, the number of all pairwise nonisomorphic rectangular nontrivial dimonoids of order~$n$ is given by $\mathrm{rdm}(n)\! - \!\mathrm{rs}(n)$.

According to Corollary~\ref{char_iso_com_rec_dm}, the numbers of pairwise nonisomorphic noncommutative rectangular dimonoids and pairwise nonisomorphic noncommutative  nontrivial rectangular dimonoids of order~$n$ are equal to $\mathrm{rdm}(n)\! -\! 1$ and $\mathrm{rdm}(n)\! - \!\mathrm{rs}(n)$, respectively.

\begin{problem}Determine the numbers $\mathrm{dm}(n)$ for $n\geq 6$ and $\mathrm{cdm}(n)$, $\mathrm{adm}(n)$, $\mathrm{rdm}(n)$ for $n\geq 7$.
\end{problem}


\appendix

\section{Program code for computing nonisomorphic dimonoids}\label{appnd}

This \texttt{Python} code processes tables of all nonisomorphic semigroups of order $n$ that were previously generated using \texttt{GAP}, where the element numbering was adjusted from $\{1,\ldots,n\}$ to $\{0,\ldots,n\!-\!1\}$. It loads these tables from \texttt{csv}-files, generates all their permutations, and checks pairwise combinations for compliance with the system of dimonoid axioms. It then eliminates isomorphic cases, retaining only representatives of isomorphism classes, and produces a complete list of pairwise nonisomorphic dimonoids of a given order $n$. The results are saved as operation tables for $\dashv$ and $\vdash$ into a single \texttt{csv}-file for further analysis. For $n=6$, we translated this code from \texttt{Python} into \texttt{C++} and performed parallel computations. We present the \texttt{Python} version here, as it is more readable and accessible to the reader.

\begin{lstlisting}
import csv
import glob
import os
from itertools import permutations, product

# ----------------------------
# 1) Load semigroup tables
# ----------------------------
semigroup_tables = []
n = None  # size will be determined automatically

for filename in glob.glob("semigroup_*.csv"):
    with open(filename, newline='') as csvfile:
        reader = csv.reader(csvfile)
        table = []
        for row in reader:
            nums = [int(cell) for cell in row if cell.strip() != ""]
            if n is None:
                n = len(nums)  # determine size from the first row
            table.append(nums)
        semigroup_tables.append(table)

print(f"Loaded {len(semigroup_tables)} semigroup tables (order {n})")

# ----------------------------
# 2) Generate all permutations for each table
# ----------------------------
all_tables = []
for tab in semigroup_tables:
    for p in permutations(range(n)):
        inv = [0]*n
        for idx, val in enumerate(p):
            inv[val] = idx
        newtab = [[p[tab[inv[i]][inv[j]]] for j in range(n)] for i in range(n)]
        all_tables.append(newtab)

print(f"Total associative tables with permutations: {len(all_tables)}")

# ----------------------------
# 3) Check dimonoid axioms
# ----------------------------
def is_dimonoid(op_dashv, op_vdash):
    for x, y, z in product(range(n), repeat=3):
        if op_dashv[op_dashv[x][y]][z] != op_dashv[x][op_vdash[y][z]]: return False #D1
        if op_dashv[op_vdash[x][y]][z] != op_vdash[x][op_dashv[y][z]]: return False #D2
        if op_vdash[op_dashv[x][y]][z] != op_vdash[x][op_vdash[y][z]]: return False #D3
        #if op_dashv[x][y] != op_dashv[y][x]: return False #CommL
        #if op_vdash[x][y] != op_vdash[y][x]: return False #CommR
        #if op_vdash[x][y] != op_dashv[y][x]: return False #Abelian
        #if op_dashv[x][op_dashv[y][z]] != op_dashv[x][z] or op_vdash[x][op_vdash[y][z]] != op_vdash[x][z]: return False #Rectangular
    return True

# ----------------------------
# 4) Find all dimonoids
# ----------------------------
dimonoid_pairs = []
for t1 in all_tables:
    for t2 in all_tables:
        if is_dimonoid(t1, t2):
            dimonoid_pairs.append((t1, t2))

print(f"Found {len(dimonoid_pairs)} ordered dimonoid pairs")

# ----------------------------
# 5) Canonization under permutations
# ----------------------------
canon_seen = set()
rep_pairs = []  # list of representatives to output
for t1, t2 in dimonoid_pairs:
    keys = []
    for p in permutations(range(n)):
        inv = [0]*n
        for idx, val in enumerate(p):
            inv[val] = idx
        r1 = tuple(tuple(p[t1[inv[i]][inv[j]]] for j in range(n)) for i in range(n))
        r2 = tuple(tuple(p[t2[inv[i]][inv[j]]] for j in range(n)) for i in range(n))
        keys.append((r1, r2))
    min_key = min(keys)
    if min_key not in canon_seen:
        canon_seen.add(min_key)
        rep_pairs.append(min_key)

print(f"Number of nonisomorphic dimonoids (order {n}): {len(canon_seen)}")

# ----------------------------
# 6) Save all nonisomorphic dimonoids into one csv-file
# ----------------------------
out_dir = os.path.dirname(os.path.abspath(file))  # program directory
list_file = os.path.join(out_dir, f"list_dm{n}.csv")

with open(list_file, "w", newline='', encoding="utf-8") as f:
    writer = csv.writer(f)
    for idx, (r1, r2) in enumerate(rep_pairs, start=1):
        writer.writerow([f"Dimonoid #{idx}"])
        writer.writerow(["Operation dashv"])
        writer.writerows(r1)
        writer.writerow([])  # empty line
        writer.writerow(["Operation vdash"])
        writer.writerows(r2)
        writer.writerow([])  # separator
        writer.writerow([])

print(f"\nAll nonisomorphic dimonoids saved to {list_file}")
\end{lstlisting}

\end{document}